\documentclass[12pt,english]{smfart}

\usepackage[T1]{fontenc}
\usepackage[english,french]{babel}

\usepackage{amsfonts}
\usepackage{amssymb,url,xspace,smfthm}
\usepackage{csquotes}
\usepackage{stmaryrd}
\usepackage[all]{xy}
\usepackage{enumitem}
\usepackage{hyperref}
\usepackage{amsmath}

\usepackage{tikz}

\newcommand\numberthis{\addtocounter{equation}{1}\tag{\theequation}}

\newtheorem{theorem}{Theorem}[section]
\newtheorem{proposition}[theorem]{Proposition}
\theoremstyle{remark}
  \newtheorem{remark}[theorem]{Remark}
  
\theoremstyle{definition}
  \newtheorem{definition}[theorem]{Definition}

\theoremstyle{definition}
    
\theoremstyle{definition}
    
\theoremstyle{definition}
    \newtheorem{prop-def}[theorem]{Proposition-Definition}

\theoremstyle{definition}
    \newtheorem{lemma}[theorem]{Lemma}

\newcommand{\LL}{\mathbb{L}}

\newcommand{\R}{\mathbb{R}}
\newcommand{\PP}{\mathbb{P}}
\newcommand{\C}{\mathbb{C}}
\newcommand{\Q}{\mathbb{Q}}
\newcommand{\N}{\mathbb{N}}
\newcommand{\Z}{\mathbb{Z}}
\newcommand{\ac}{\overline{ac}}
\newcommand{\nash}{\hat{X}}
\newcommand{\field}{\mathbb{C}(\!(t)\!)}
\newcommand{\sheaf}{\mathcal{O}_{\hat{X}}(1)}

\title[Motivic local density of isolated singularities of alg. var.]{Motivic local density of isolated singularities of algebraic varieties}
\author{Sidonie Ratajczak}
\address{Univ. Lille, CNRS, UMR 8524 - Laboratoire Paul Painlevé, F-59000 Lille, France}

\begin{document}

\begin{abstract}
The goal of this paper is to compute the motivic local density of an isolated singularity of an algebraic variety, in order to explain its link with algebraic multiplicity. We compute the density as the residue in 1 of a Zeta function associated to normalized local volumes and we use a resolution of the singularity to get the formula. 
\end{abstract}

\maketitle

\tableofcontents

\section{Introduction}

Non-Archimedean geometry has emerged as a powerful setting to define and compare additive invariants attached to singularities of algebraic varieties. In particular, we study in this paper the motivic local density. This notion had been developed in different settings, defined as limits of normalized local volumes, but one can also compute local density using a tangent cone with multiplicities. 

We consider $K$ a field equipped with a distance and a measure $\mu$, the local density of a set $X \subset K^n$ of dimension $d$ at a point $x$ is defined as the following limit, if it exists: 
$$
\Theta_d(X,x) = \lim_{r \longrightarrow 0} \frac{\mu_d(X \cap B_n(x,r))}{\mu_d(B_d(x,r))}.
$$

In \cite{lelong_integration_1957}, Lelong first shows the existence of the limit if $X$ is a complex analytic manifold, and Thie in \cite{thie_lelong_1967} expresses the local density as a sum of the local densities of components of the tangent cone of $X$ at $x$ with multiplicities. The latter implies that $\Theta_d(X,x)$ is a positive integer. In \cite{kurdyka_densite_1989}, Kurdyka and Raby show the existence of the local density if $X$ is real subanalytic. They prove an analogue of Thie’s result by expressing it as a finite sum of densities of components of the positive tangent cone, with multiplicities, but the local density is no longer an integer.

\vspace{\baselineskip}

Cluckers, Comte and Loeser show in \cite{cluckers_local_2012} the existence of a notion of local density for a semi-algebraic or subanalytic subset $X \subset K^n$, where $K$ is a finite extension of $\Q_p$. The definition is slightly different than just the analogous of the real one, with Lebesgue measure replaced by Haar measure, because the limit does not exist in general. The sequence of normalized local volumes has a cyclical convergence, they define the local density as the mean value of the different limits. In this $p$-adic setting, there is no natural notion of tangent cone, therefore they
are led to study for any $\Lambda$, a multiplicative subgroup of $K^\times$ of finite index, a $\Lambda-$tangent cone. They show the existence of a distinguished tangent cone. Then they assign multiplicities to the tangent cone and show as in the real case that the local density can be computed on this tangent cone with multiplicities.

\vspace{\baselineskip}

These results have been generalized to the case of the Laurent series $\field$, and more generally to discrete valued Henselian fields of characteristic zero by A. Forey \cite{forey_motivic_2017}, using motivic integration to give a sense to the volume. It is not a classical measure theory as Lebesgue, the volume is no longer a real number but an element of the Grothendieck ring of varieties on the residue field, denoted $K_0(\text{Var}_k)$. Denote $\mu_d$ the $d$-dimensional measure defined in \cite{cluckers_constructible_2008}, and let $B_n(x,m) = \{y \in k(\!(t)\!)^n \mid v(x-y) \geq m\}$ the $n$-dimensional ball around $x \in k(\!(t)\!)^n$ of valuative radius $m \in \Z$, for $k$ an algebraically closed field of characteristic zero. Consider $X$ a variety over $k$ of dimension $d$ with an isolated singularity $O$, and $\underline{X} \subset h[n,0,0]$ the definable set associated to the $k[\![t]\!]$-points of $X$. We have: 
$$
\theta_m := \frac{\mu_d(\underline{X}\cap B_n(x,m))}{\mu_d(B_d(x,m))} \in \mathcal{M}_k
$$
where $\mathcal{M}_k$ denotes a localization of the Grothendieck ring (See Section \ref{section : mot int}), and there exists an integer $e > 0$ such that for any $i = 0, 1, ..., e - 1$, the subsequences $(\theta_{ke+i})_{k \geq 0}$ converge to some $d_i \in \mathcal{M}_k$, where the topology is induced by the degree in $\LL$. As in the $p$-adic case, we define the motivic local density as the mean value of the $d_i$. Since the motivic local density is independent of the choice of the embedding (see Proposition \ref{prop : indep embed}), we have that:
$$
\Theta_d^\text{mot}(X,O):= \frac{1}{e} \sum_{i=0}^{e-1} d_i \in \Q \otimes \mathcal{M}_k.
$$

In this case, we also consider $\Lambda$-tangent cones for all $\Lambda$ of the form $\Lambda_n=\{\lambda \in k(\!(t)\!)^\times \mid \ac(\lambda)=1, \text{ord}(\lambda) \equiv 0 \text{ mod } n\}$. There is a distinguished tangent cone, and the local density of $X$ at $x$ can be computed after assigning multiplicities to this cone.

\vspace{\baselineskip}

In this paper, we prove that the motivic local density can be computed as the residue in 1 of a Zeta function defined with the normalized local volumes. Let $\mathcal{M}_k=K_0(\text{Var}_k)\Big[\LL^{-1},\Big\{\frac{1}{1-\LL^{-i}}\Big\}_{i>0}\Big]$, we consider the ring $R:=\Q \otimes \mathcal{M}_k[T,(1-\LL^{-a}T^b)^{-1}]$, $a\in \Z, b \geq 1$. We know by rationality results (See for example \cite{cluckers_constructible_2008}) that for a constructible motivic function with parameter $n$, the power serie lies in this ring. We then define $S$ the subset of $R$ consisting of elements of the form $\frac{P}{Q}$ with $P \in \mathcal{M}_k[T], Q = \prod_{i=1}^r(1-\LL^{-a_i}T^{b_i})$ such that at most one of the $a_i$ equal 0. Informally, this ring consists of rational fractions with a pole of order at most 1 in 1. We define $\text{Res}_{T=1}(\frac{P}{Q})=\text{Ev}_{T=1}((1-T)\frac{P(T)}{Q(T)})$, see Definition \ref{def: residue}.

\vspace{\baselineskip}

\begin{theorem}\label{thm: intro zeta}
    Let $D$ be a definable set of $h[m,0,0]$ of dimension $d$. The motivic local density of $D$ at a point $x$ can be computed as the residue in 1 of a Zeta function of normalized local volumes. More precisely, if one sets
    $$
    Z(T) = \sum_{n \geq 0} \frac{\mu_d(D \cap B_m(x,n))}{\mu_d(B_d(x,n))}T^n,
    $$
    then $Z(T) \in S$ and
    $$
    \Theta_d^\text{mot}(D,x)=\mathrm{Res}_{T=1}(Z(T)).
    $$ 
\end{theorem}

\vspace{\baselineskip}

We use this theorem to prove an explicit formula for the motivic local density of isolated singularities of algebraic varieties of any dimension. The method is based on a resolution of singularities. The morphism of resolution allows us to compute the motivic local density using a theorem of change of variable (see Section \ref{section : mot int}, Theorem \ref{thm: change var}). The jacobian factor appearing in this change of variable can be expressed with the Mather discrepancies (see Section \ref{section: mather disc}). A resolution of $(X,0)$ is good if its exceptional locus is a normal crossing divisor. We also require that the resolution factors through the blowup of the maximal ideal and through the Nash modification of $(X,O)$. The main result of this paper is the following theorem:

\vspace{\baselineskip}

\begin{theorem}\label{thm: intro formula}
    Let $X$ be an algebraic variety of dimension $d$ over $\C$ with a unique isolated singularity $O$. Let $\pi : Y \longrightarrow X$ be a good resolution of the singularity, that satisfies the previous conditions. Set $E_i, i \in J$ the irreducible components of the exceptional locus, and $m_i$ its multiplicity. For $i \in J$ and $I \subset J$, we consider the schemes $E_i^0=E_i \smallsetminus \cup_{j \neq i} E_j, E_I=\cap_{i\in I} E_i$ and $E_I^0=E_I \smallsetminus \cup_{j \in J\smallsetminus I} E_j$.
    
    The motivic local density of $X$ at the point $O$ is given by:
    $$
    \Theta^\text{mot}_d(X,O) = \sum_{i\mid \alpha_i=0} \frac{1}{m_i} \sum_{\substack{I\subset J\\ I=\{i,\dots,i_s\}}} (\LL-1)^{|I|}\prod_{j \in I\setminus\{i\}}  \frac{\LL^{-\alpha_{j}}}{1-\LL^{-\alpha_{j}}}\frac{[E_I^0]}{\LL^d-1} ,
    $$
    which lies in $\mathcal{M}_\C \otimes \Q$, and where $\alpha_i=\hat{k}_i+1-dm_i$, with $\hat{k}_i$ the Mather discrepancy. 
\end{theorem}

\vspace{\baselineskip}

This result is a generalization of the formula for surfaces proved in \cite{ratajczak_motivic_2026}. In this case, data are used: the inner rate related to the bilipschitz geometry of the singularity. Those inner rates were first introduced in \cite{birbrair_thick-thin_2014} by L. Birbrair, W. D. Neumann and A. Pichon, and were studied in \cite{silva_inner_2022} by A. Belotto da Silva, L. Fantini and A. Pichon. They study the metric structure of the germ of an isolated complex surface singularity $(X,O)$ by means of the infinite family of numerical analytic invariants, called inner rates. In this context, we also consider $l:(X,O) \longrightarrow (\C^2,0)$ a generic projection, and a good resolution of the singularity such that $\Pi_l^* \cap E_j = \emptyset$ for components $E_j$ with associated inner rate equal to 1, where $\Pi_l^*$ denotes the strict transform of the polar curve $\Pi_l$.

\vspace{\baselineskip}

\begin{theorem}\label{thm : formula surface}
    Let $X$ be an algebraic surface (i.e. a 2-dimensional separated scheme of finite type) over $\C$ with a unique isolated singularity $O$. Let $\pi : Y \longrightarrow X$ be a good resolution of the singularity that satisfies the previous conditions. Denote $E_i$ the irreducible components of the exceptional locus, $m_i$ its multiplicity and $q_i$ its associated inner rate. Set $E_i^0=E_i \smallsetminus \cup_{j \neq i} E_j$.
    The motivic local density of $X$ at the point $O$ is given by:
    $$
    \Theta^\text{mot}_2(X,O) = \sum_{i\mid q_i=1} \frac{1}{m_i} \Bigg(\frac{1}{\LL+1}[E_i^0] + \sum_{j \mid E_i\cap E_j \neq \emptyset }\frac{\LL^{-(q_{j}-1)m_{j}}(1-\LL^{-1})}{(1-\LL^{-(q_{j}-1)m_{j}})(1+\LL^{-1})} \Bigg),
    $$
    which lies in $\mathcal{M}_\C \otimes \Q$.
\end{theorem}

\vspace{\baselineskip}

The paper is organized as follows. In the next section, we will recall some basic definitions on motivic integration as well as properties of the motivic local density. The Section \ref{section: zeta function} is devoded to the proof of Theorem \ref{thm: intro zeta}. We then provide the definition of the Mather discrepancy divisor in Section \ref{section: mather disc}. We prove Theorem \ref{thm: intro formula} in Section \ref{section: proof thm} and we also explain how to recover the formula for surfaces using our new formula. The last section of this paper is devoded to the definition of a motivic Zeta function associated to a sheaf of ideals and a singular algebraic variety over a field of characteristic zero. From this definition, one can recover the motivic Zeta functions already studied in various contexts that we will recall, and also the Zeta function of normalized local volumes created to compute the motivic local density.

\vspace{\baselineskip}

\textbf{Acknowledgements:} I acknowledge the support of the CDP C2EMPI, together with the French State under the France-2030 programme, the University of Lille, the Initiative of Excellence of the University of Lille, the European Metropolis of Lille for their funding and support of the R-CDP-24-004-C2EMPI project. I gratefully acknowledge R. Cluckers and A. Forey for their guidance, and L. Fantini and A. Pichon for enlightening discussions. 

\section{Recap on motivic integration and local density} \label{section : mot int}

We are considering the theory of Cluckers-Loeser (See \cite{cluckers_constructible_2008}), specialized to the case of algebraically closed residue field. Let us recall the setting. Fix $k$ a field of characteristic zero. Consider the three sorted language of Denef-Pas, with one sort for the valued field, with the language of rings $\{0,1, +,-, \cdot\}$, one sort for the residue field with the language of rings and one sort for the value group with the Presburger language $\{0,+,\leq,\{\equiv_n\}_{n \geq 1} \}$ where $\equiv_n$ is a binary relation for congruence modulo $n$. We also add a symbol $v$ for the valuation map, $\ac$ for the angular component, and constants symbols for elements of $k(\!(t)\!)$. We will use this language to describe the Henselian discretely valued field of characteristic zero. We will consider structures for $\mathcal{L}_{\text{DP}}$ consisting of tuples $(K,k,\Gamma)$ where $K$ is a
valued field with value group $\Gamma$, and residue field $k$. For example, for all field $k$, $k(\!(t)\!)$  the field of formal Laurent series over $k$ is an $\mathcal{L}-$structure $(k(\!(t)\!),k,\Z)$, where the symbols are naturally interpreted: 
$$
v\left(\sum_{i \geq i_0}^\infty  a_it^i \right) = i_0, ~ \ac \left( \sum_{i \geq i_0}^\infty  a_it^i \right) = a_{i_0} \text{ if } a_{i_0} \neq 0.
$$

The $\mathcal{L}_\text{DP}$-theory of the above discribed structures whose valued field is Henselian and whose residue field is of characteristic zero admits elimination of quantifiers in the valued field sort. 

\vspace{\baselineskip}

\begin{definition}(See Section 2.3 of \cite{cluckers_constructible_2008})
    We define $h[m,n,r]$ the functor from the category of algebraically closed fields containing $k$ to the category of sets sending an algebraically closed field $F$ to 
    $$
    h[m,n,r](F)=F(\!(t)\!)^m \times F^n \times \Z^r.
    $$
    For $X$ an affine variety over $F$ in $\mathbb{A}^n_F$, denote $\underline{X}$ the subdefinable set of $h[n,0,0]$, which to any algebraically closed field $F$ associates $X(F[\![t]\!])$.
\end{definition}

\vspace{\baselineskip}

In this paper, definable means definable in the Denef-Pas language. \\

\begin{definition}
    If $D$ is a subassignment of $h[m,n,r]$, we define $\text{dim}(D)$ as the dimension of the image of $D$ under the projection $h[m,n,r] \longrightarrow h[m,0,0]$.
\end{definition}

\vspace{\baselineskip}

\begin{definition}
    Let $k$ be a field, the Grothendieck group of varieties $K_0(\text{Var}_k)$ is the free abelian group generated by symbols $[X]$ for the isomorphism class of $X$ a variety (a separated scheme of finite type over $k$), with relations  $[X] = [Y]+[X \smallsetminus Y]$, where $Y$ is a closed subvariety. The fiber product induces a ring structure on $K_0(\text{Var}_k)$ with identity element the class of a point. Set $\LL:=[\mathbb{A}_k^1]$ the class of the affine line.
\end{definition} 



\vspace{\baselineskip}

We consider the localization
$$
\mathcal{M}_k := K_0(\text{Var}_k)\left[ \LL^{-1}, \left\{ \frac{1}{1-\LL^{-i}} \right\}_{i >0} \right].
$$
The topology on $\mathcal{M}_k$ is induced by the $\LL-$degree. \\

In \cite{cluckers_constructible_2008}, Cluckers and Loeser assign to some definable subset $D$ of dimension $d$ an element of $\mathcal{M}_k$, called the motivic volume of $D$, that we denote by $\mu_d(D)$. They define more generally a whole class of motivic constructible functions that are stable under integration and with good properties like Fubini’s theorem and change of variable formula. 

\vspace{\baselineskip}

Let $h$ be a definable subassignment of $h[m,n,r]$ of dimension $d$. We denote by $x_1,\dots,x_m$ the coordinates on $\mathbb{A}^m_{k}$ and we consider the $d$-forms $\omega_I := dx_{i_1} \wedge \dots \wedge dx_{i_d}$ for $I=\{i_1,\dots,i_d\} \subset \{1,\dots,m\}$, $i_1<\dots<i_d$. We denote by $|\omega_I|_h$ the image of $\omega_I$ in $|\tilde{\Omega}|_+(h)$. We recall the following definitions from \cite{cluckers_constructible_2008}.

\vspace{\baselineskip}

\begin{definition}[8.3.1 of \cite{cluckers_constructible_2008}]
    There is a unique element $|\omega_0|_h$ in $|\tilde{\Omega}|_+(h)$, the canonical volume form, such that, for every $I$, there exist $\Z$-valued definable functions $\alpha_I$ and $\beta_I$ on $h$, with $\beta_I$ only taking as values 1 and 0, such that $\alpha_I + \beta_I >0$ on $h$, $|\omega_I|_h=\beta_I\LL^{-\alpha_I}|\omega_0|_h$ in $|\tilde{\Omega}|_+(h)$, and such that $\text{inf}_I\alpha_I=0$. We call $|\omega_0|_h$ the canonical volume form on $h$.\label{def : volume form embed}
\end{definition}

\vspace{\baselineskip}

\begin{theorem}(Theorem 12.1.1 of \cite{cluckers_constructible_2008}, or Lemma 3.3 of \cite{denef_germs_1999})\label{thm: change var}
    Let $h : Y \longrightarrow X$ a proper birational morphism between two algebraic varieties over $k$ of pure dimension $d$, consider $A$ a definable set of $\underline{X}$, $\alpha : A \longrightarrow \N$ a simple function. Then we have the following equality:
    $$
    \int_{A} \LL^{-\alpha}|\omega_X| = \int_{h^{-1}(A)} \LL^{-\alpha \circ h - v(h^*(\Omega_X^d))}|\omega_Y|,
    $$
    where $|\omega_X|,|\omega_Y|$ are the canonical volume forms associated to $\underline{X}$ and $\underline{Y}$.
\end{theorem}

\vspace{\baselineskip}

\begin{remark}
    When $X$ is a smooth $k-$variety of pure dimension $d$, then $\mu_d(\underline{X})=[X]\LL^{-d}$. 
\end{remark}

\vspace{\baselineskip}

We briefly recall the definition of the motivic local density, and some of its properties. The local density had been studied in the case of the Laurent series field $\field$, and more generally in the case of discrete valued Henselian fields of characteristic zero by A. Forey in \cite{forey_motivic_2017}. 

\vspace{\baselineskip}

\begin{definition}
    Say that a function $h : \mathbb{N} \longrightarrow \mathcal{M}_k$ has a mean value at infinity if there exists an integer $e > 0$ such that 
    $$
    \lim\limits_{\overset{n \longrightarrow \infty}{n \equiv c \text{ mod } e}} h(n)
    $$
    exists in $\R$ for each $c=0,\dots,e-1$ and in this case define the mean value at infinity of $h$ as the average 
    $$
    \text{MV}_\infty(h):=\frac{1}{e} \sum_{i=0}^{e-1} \lim\limits_{\overset{n \longrightarrow \infty}{n \equiv c \text{ mod } e}} h(n).
    $$
    The mean value is independent of the choice of the modulus $e>0$. We recall that the topology on $\mathcal{M}_k$ is induced by the $\LL-$degree.
\end{definition}

\vspace{\baselineskip}

Consider $D \in h[m,0,0]$ a definable set of dimension $d$. Consider $x$ a point of $D$ and set $\theta_d(D,x)(n):= \frac{\mu_d(D \cap B_m(x,n))}{\mu_d(B_d(x,n))} \in \mathcal{M}_k$. 

\vspace{\baselineskip}

\begin{definition}\label{def:local density}
    The function $\theta_d(D,x)$ admits a mean value at infinity, and we define the local density of the set $D$ at the point $x$ by: 
    $$
    \Theta_d(D,x):= \text{MV}_\infty(\theta_d(D,x)) \in \Q \otimes \mathcal{M}_k.
    $$
    (See Proposition-Definition 3.6 of \cite{forey_motivic_2017} for more details.)
\end{definition}

\vspace{\baselineskip}

With the properties of the motivic measure, we trivially have the following proposition: 

\vspace{\baselineskip}

\begin{proposition}
    Consider $D \subset h[m,0,0]$ a definable set of dimension $d$, $U \subset D$ a subset of $D$ of dimension $<d$ and $x$ a point of $D$. Hence
    $$
    \Theta_d(D,x) = \Theta_d(D \smallsetminus U,x).
    $$
\end{proposition}

\vspace{\baselineskip}

The following proposition shows that renormalizing with balls instead of with spheres yields the same local density. 

\vspace{\baselineskip}

\begin{proposition} \label{prop:normalization balls or spheres}
    Let $D \subset h[m,0,0]$ be a definable set of dimension $d$ and $x$ be a point of $D$. We have the following equality:
    \begin{align*}
    \Theta_d(D,x) & = \text{MV}_\infty\Big(n\mapsto \frac{\mu_d(D\cap B_m(x,n))}{\mu_d(B_d(x,n))}\Big) \\
    &= \text{MV}_\infty\Big(n\mapsto \frac{\mu_d(D\cap S_m(x,n))}{\mu_d(S_d(x,n))}\Big).
    \end{align*}
\end{proposition}

\vspace{\baselineskip}

We also recall the following results:

\vspace{\baselineskip}

\begin{proposition}(Proposition 5.7 of \cite{ratajczak_motivic_2026})\label{prop : indep embed}
    Consider $X$ a variety over $k$ of dimension $d$ with an isolated singularity $O$. The motivic local density $\Theta_d(\underline{X},O)$ is independent of the choice of an embedding in the affine space, and we write it $\Theta_d^\text{mot}(X,O)$. 
\end{proposition}

\section{Zeta function}\label{section: zeta function}

In this section, we define a notion of residue for a rational fraction with coefficients in $\mathcal{M}_k$. We then define a Zeta function associated to normalized local volumes of a definable set and prove that we can compute its residue in 1, and the latter is equal to the motivic local density.

\vspace{\baselineskip}

\begin{definition}\label{def: residue}
    Let $\mathcal{M}_k=K_0(\text{Var}_k)\Big[\LL^{-1},\Big\{\frac{1}{1-\LL^{-i}}\Big\}_{i>0}\Big]$, and let $R:=\Q \otimes \mathcal{M}_k[T,(1-\LL^{-a}T^b)^{-1}]$, $a\in \Z, b \geq 1$. We define $S$ the subset of $R$ consisting of elements of the form $\frac{P}{Q}$ with $P \in \mathcal{M}_k[T], Q = \prod_{i=1}^r(1-\LL^{-a_i}T^{b_i})$ such that at most one of the $a_i$ equal 0. Informally, this ring consists of rational fractions with a pole of order at most 1 in 1.

    If exactly one $a_i=0$, then:
    \begin{align*}
        (1-T)\frac{P(T)}{Q(T)} = & \frac{P(T)}{1+T+\dots + T^{b_i-1}} \times \frac{1}{\prod_{k \neq i}(1-\LL^{-a_k}T^{b_k})} \\
        =& \frac{P(T)}{\tilde{Q}(T)}
    \end{align*}
    We have that $\tilde{Q}(1)=b_i\prod_{k\neq i}(1-\LL^{-a_k}) \neq 0$ and $\tilde{Q}(1)$ invertible in $\Q \otimes \mathcal{M}_k$. We define the following $\Q \otimes \mathcal{M}_k$-linear map:
    
    \begin{center}
    \begin{tabular}{ccccc}
        $\text{Res}_{T=1} :$ & $S$ & $\longrightarrow$ & $\Q \otimes \mathcal{M}_k$ \\
         & $\frac{P(T)}{Q(T)}$ & $\mapsto$ & $\text{Ev}_{T=1}((1-T)\frac{P(T)}{Q(T)}) = \frac{P}{\tilde{Q}}(1)$\\
    \end{tabular}
    \end{center}

    If none of the $a_i =0$, then $\text{Res}_{T=1}\Big(\frac{P(T)}{Q(T)}\Big)=0$.

    The definition of the map $\text{Res}_{T=1}$ does not depend on the representation. 
\end{definition}

\vspace{\baselineskip}

\begin{theorem}\label{thm: zeta local volumes}
    Let $D$ be a definable set of $h[m,0,0]$ of dimension $d$. The motivic local density of $D$ in a point $x$ can be computed as the residue in 1 of a Zeta function of normalized local volumes. More precisely, if one sets
    $$
    Z(T) = \sum_{n \geq 0} \frac{\mu_d(D \cap B_m(x,n))}{\mu_d(B_d(x,n))}T^n ,
    $$
    then $Z(T) \in S$ and
    $$
    \Theta_d^\text{mot}(D,x)=\mathrm{Res}_{T=1}(Z(T)).
    $$ 
\end{theorem}

\vspace{\baselineskip}

\begin{remark}
    One could also define a Zeta function as 
    $$
    Z(T)=\sum_{n \geq 0} \frac{\mu_d(D\cap S_m(x,n))}{\mu_d(S_d(x,n))}T^n 
    $$
    since this choice of local volumes leads to the same local density, as stated in Proposition \ref{prop:normalization balls or spheres}.
\end{remark}

\vspace{\baselineskip}

\begin{proof} [Proof of Theorem \ref{thm: zeta local volumes}]
    Set $a_n:=\theta_d(D,x)=\frac{\mu_d(D \cap B_m(x,n))}{\mu_d(B_d(x,n))}$. By Definition \ref{def:local density}, there exists an integer $e$ such that the subsequences $(a_{ke+i})_k$ converge to $d_i \in \mathcal{M}_\C$ for $i=0,\dots,e-1$. Set $\varepsilon_{k,i}:=a_{ke+i}-d_i$, hence:
    \begin{align*}
        Z(T)=& \sum_{i=0}^{e-1}\sum_{k \geq 0}d_iT^{ke+i}+\sum_{i=0}^{e-1}\sum_{k \geq 0}\varepsilon_{k,i}T^{ke+i} \\
        =& \sum_{i=0}^{e-1}d_i \frac{T^i}{1-T^e}+\sum_{i=0}^{e-1}\sum_{k \geq 0}\varepsilon_{k,i}T^{ke+i}.
    \end{align*}
    Clearly $\text{Res}_{T=1}\Big(\sum_{i=0}^{e-1}d_i \frac{T^i}{1-T^e}\Big)=\frac{1}{e}\sum_{i=0}^{e-1}d_i$. 

    By Presburger cell decomposition, one can assume that there is $k \in \mathbb{N}_{>0}$, $b \in \N$, some elements $h_{j,i}$ of $\mathbb{K}(\text{Var}_k)$, $\alpha_j \in \mathbb{N}, \beta_j \in \Z$ for $j=1,\dots,k$ and $i=0,\dots,e-1$ such that for all $n \geq b$, $n \equiv i \text{ mod } e$, one has
    $$
    a_{n}=\sum_{j=1}^kh_{j,i} \Big(\frac{n-i}{e}\Big)^{\alpha_j}\LL^{\beta_j\Big(\frac{n-i}{e}\Big)}.
    $$
    One can moreover assume that the pairs $(\alpha_j,\beta_j)$ are pairwise distinct, and the functions $h_{i,j}$ are not identically zero. For all $j, \beta_j \leq 0$ and if $\beta_j=0$, then $\alpha_j=0$. Let $j_0$ such that $(\alpha_{j_0},\beta_{j_0})=(0,0)$. One then has $\lim_{k \rightarrow+\infty} a_{ke+i}=h_{j_0,i}=d_i$. This implies that :
    \begin{align*}
        \varepsilon_{k,i} = \sum_{j \neq j_0}h_{j,i}k^{\alpha_j}\LL^{\beta_jk} +d_i-d_i.
    \end{align*}
    We can now compute : 
    \begin{align*}
        \sum_{i=0}^{e-1} \sum_{k \geq 0} \varepsilon_{k,i}T^{ke+i} &= \sum_{i=0}^{e-1} \sum_{k \geq 0} \sum_{j\neq j_0}h_{j,i}k^{\alpha_j}\LL^{\beta_jk} T^{ke+i} \\
        & = \sum_{i=0}^{e-1}T^i\sum_{j \neq j_0} h_{j,i} \sum_{k\geq 0} k^{\alpha_j}T^{ke} \LL^{\beta_jk} \\
        & = \sum_{i=0}^{e-1} T^i \sum_{j \neq j_0} h_{j,i} \frac{Q(T)}{(1-T^e\LL^{\beta_j})^{\delta+1}},
    \end{align*}
    where $Q(T) \in \mathcal{M}_\C[T]$ is a polynomial of degree $\delta$. Since in the previous expression, $\beta_j <0$ for all $j$, then it satisfies the condition of the definition \ref{def: residue}, and hence 
    $$
    \text{Res}_{T=1}\Big(\sum_{i=0}^{e-1} \sum_{k \geq 0} \varepsilon_{k,i}T^{ke+i}\Big)=0.
    $$
    This proves that $Z(T) \in S$ and $\text{Res}_{T=1}(Z(T)) = \frac{1}{e}\sum_{i=0}^{e-1} d_i = \Theta_d^\text{mot}(D,x)$. 
\end{proof}

\section{The Mather discrepancy} \label{section: mather disc}

In this section, we will recall the definition of the Mather discrepancy divisor, that we will use to compute the order of the jacobian that appears in the change of variables formula while computing the local density via a morphism of resolution. We set $k=\C$ in this section. Given an arbitrary variety $X$, we take a resolution of singularities $\pi : Y\longrightarrow X$ that factors through the Nash blow-up and define the relative Mather canonical divisor $\hat{K}_{Y/X}$.  
We recall the following definitions, assuming that all varieties are irreducible and reduced, over a fixed algebraically closed field k of characteristic zero, see \cite{ein_multiplier_2011}, \cite{ishii_mather_2011} for more details. \\

Let $X$ be a $\Q-$Gorenstein variety of index $r$ (i.e. normal and its canonical class $K_X$ is $\Q-$Cartier, so $rK_X$ is Cartier), and $\pi : Y \longrightarrow X$ a resolution of singularities of $X$. Then the usual discrepancy divisor $K_{Y/X}$ is the unique $\Q$-Cartier divisor supported on the exceptional locus of $\pi$ such that $rK_{Y/X}$ is linearly equivalent with $rK_Y-\pi^*(rK_X)$. Note that the usual discrepancy is defined only for a $\Q-$Gorenstein variety $X$, and the following Mather discrepancy is defined for every variety, even for non-normal variety. 

\vspace{\baselineskip}

\begin{definition}
    Let $X$ be a $d$-dimensional variety. The sheaf $\Omega_X^d = \wedge^d\Omega_X$ is invertible over the smooth locus $X_\text{reg}$ of $X$, hence the morphism 
    $$
    \pi : \PP(\Omega_X^d) = \text{Proj}(\text{Sym}(\Omega_X^d)) \longrightarrow X
    $$
    is an isomorphism over $X_\text{reg}$. The \textit{Nash blow-up} $\nash$ is the closure of $\pi^{-1}({X_\text{reg}})$ in $\PP(\Omega_X^d)$ (with the reduced scheme structure). Note that by construction we have a projective birational morphism $\nu : \nash \longrightarrow X$ and a surjective morphism $\nu^*(\Omega_X^d) \longrightarrow \mathcal{O}_{\PP(\Omega_X^d)}(1)_{|\nash}$. Set $\sheaf := \mathcal{O}_{\PP(\Omega_X^d)}(1)_{|\nash}$, it is an ample line bundle over $\hat{X}$. Furthermore, this is universal in the following sense: given a birational morphism of varieties $\phi :Y \longrightarrow X$, this factors through $\nu$ if and only if $\phi^*(\Omega_X^d)$ admits a morphism onto an invertible sheaf $\mathcal{L}$ on $Y$ (which identifies $\mathcal{L}$ to the quotient of $\phi^*(\Omega_X^d)$ by its torsion). 
\end{definition}

\vspace{\baselineskip}

\begin{definition}\label{def:mather disc}
Let $X$ be a variety of dimension $d$ and $f : Y \longrightarrow X$ be a resolution of the singularities, factoring through the Nash blow-up. Then, the image of the canonical homomorphism 
$$
f^* \wedge^d\Omega_X \longrightarrow \wedge^d\Omega_Y
$$
is an invertible sheaf of the form $J\wedge^d\Omega_Y$, where $J$ is the invertible ideal sheaf of $\mathcal{O}_Y$ that defines an effective divisor supported on the exceptional locus of $f$. This divisor is called the \textit{Mather discrepancy divisor} and denoted by $\hat{K}_{Y/X}$. For every prime divisor $E$ on $Y$, we define 
$$
\hat{k}_E:= \text{ord}_E(\hat{K}_{Y/X})
$$
and
$$
\hat{k}_E^\text{log}:= \hat{k}_E +1.
$$
\end{definition}

\vspace{\baselineskip}

\begin{remark}
If $X$ is smooth, then $\hat{k}_E = k_E$ and $\hat{K}_{Y/X} = K_{Y/X}$. 
\end{remark}

\vspace{\baselineskip}

\begin{definition}
    The \textit{Jacobian ideal} of a variety $X$ denoted $Jac_X$ is the 0-th Fitting ideal $\text{Fitt}_0(\Omega_X)$ of $\Omega_X$. This is an ideal whose support is the
singular locus of $X$. If $f : Y \longrightarrow X$ is a log resolution of $Jac_X$, we denote by $J_{Y/X}$ the effective divisor on $Y$ such that $Jac_X \cdot \mathcal{O}_Y = \mathcal{O}_Y(-J_{Y/X} )$.
\end{definition}

\vspace{\baselineskip}

\begin{remark}[Remark 2.7 of \cite{ein_multiplier_2011}]
    If $X$ is normal and locally a complete intersection, then $\hat{K}_{Y/X}-J_{Y/X}=K_{Y/X}$ the usual discrepancy divisor.
\end{remark}

\section{Proof of the formula} \label{section: proof thm}

The goal of this section is to prove Theorem \ref{thm: intro formula}. We first recall the setting. Let $\pi : Y \longrightarrow X$ be a resolution of the singularity that factors through the blowup of the maximal ideal and the Nash transform of $(X,O)$, and such
that the exceptional locus is simple normal crossing. We denote by $E_i$ the irreducible components of the exceptional divisor $D = \pi^{-1}(O)$.

\vspace{\baselineskip}

We state the following lemma, which is the generalization in dimension $d$ of Lemma 6.3 of \cite{ratajczak_motivic_2026}, before proceeding with the proof of Theorem \ref{thm: intro formula}.

\vspace{\baselineskip}

\begin{lemma}\label{lem:vois tub}
Let $X$ be a variety over $\C$ with a unique isolated singularity in $O$, and consider a resolution $\pi : Y \longrightarrow X$ such that the exceptional locus is simple normal crossing. Then $\pi^{-1}(\underline{X} \cap S(O,n)) = \mathcal{V}(D)$ where $\mathcal{V}(D)$ is the set of points of valuative distance exactly $\frac{n}{m_i}$ of every component $E_i$ of the exceptional divisor $D=\pi^{-1}(O)=\sum_{i\in J} m_iE_i$.
\end{lemma}

\vspace{\baselineskip}

\begin{proof}
Let us set $D=\pi^{-1}(O)=\sum_{i\in J} m_iE_i$ the exceptional divisor, with $m_i$ the multiplicity of the exceptional component $E_i$. For $i \in J$ and $I \subset J$, we consider the schemes $E_i^0=E_i \smallsetminus \cup_{j \neq i} E_j, E_I=\cap_{i\in I} E_i$ and $E_I^0=E_I \smallsetminus \cup_{j \in J\smallsetminus I} E_j$. Since the divisor is supposed to be normal crossing, we can cover $Y$ by affine open subsets $U$ on which there exist regular functions $x_1,\dots,x_d$ inducing an étale map $h : U \longrightarrow \mathbb{A}_k^d$ such that on $U$, each component of the exceptional divisor is defined by a monomial in $x_1,\dots,x_d$. 

In each point of $E_I^0$, we have:
\begin{align*}
    \mathcal{V}(D)\cap \underline{U}&=\{p\in U \mid  v(\prod_{i \in I}x_i(p)^{m_i})=n\} \text{ and} \\
    \pi^{-1}(\underline{X} \cap S(O,n))\cap \underline{U} &= \{p\in U \mid \pi(p) \in S(O,n)\} \\
    & = \{p\in U \mid v(\prod_{i \in I}x_i(p)^{m_i})=n\}.
\end{align*}

\end{proof}

\vspace{\baselineskip}

\begin{theorem}\label{thm:formula}
    Let $X$ be an algebraic variety of dimension $d$ over $\C$ with a unique isolated singularity $O$. Let $\pi : Y \longrightarrow X$ be a good resolution of the singularity, that satisfies the previous conditions. Set $E_i, i \in J$ the irreducible components of the exceptional locus, and $m_i$ its multiplicity. For $i \in J$ and $I \subset J$, we consider the schemes $E_i^0=E_i \smallsetminus \cup_{j \neq i} E_j, E_I=\cap_{i\in I} E_i$ and $E_I^0=E_I \smallsetminus \cup_{j \in J\smallsetminus I} E_j$.
    
    The motivic local density of $X$ at the point $O$ is given by:
    $$
    \Theta^\text{mot}_d(X,O) = \sum_{i\mid \alpha_i=0} \frac{1}{m_i} \sum_{\substack{I\subset J\\ I=\{i,\dots,i_s\}}} (\LL-1)^{|I|}\prod_{j \in I\setminus\{i\}}  \frac{\LL^{-\alpha_{j}}}{1-\LL^{-\alpha_{j}}}\frac{[E_I^0]}{\LL^d-1} ,
    $$
    which lies in $\mathcal{M}_\C \otimes \Q$, and where $\alpha_i=\hat{k}_i+1-dm_i$. 
\end{theorem}

\vspace{\baselineskip}

\begin{proof}
The proof strategy is the following. We first use the theorem of change of variable given by the morphism $\pi$ of resolution of the singularity. We then use a Zeta function of local volumes and compute its residue in 1 to get to the motivic local density, as proved in Section \ref{section: zeta function}.

Since $\pi$ induces an isomorphism $Y \smallsetminus \pi^{-1}(O) \longrightarrow X \smallsetminus \{O\}$, we have the following equality:
    \begin{align*}
        \mu_d(\underline{X} \cap S(O,n)) & = \mu_d(\underline{X}\smallsetminus\{O\} \cap S(O,n)) \\
        & = \int_{\pi^{-1}(\underline{X}\smallsetminus\{O\} \cap S(O,n))} \LL^{-v(\pi^*(\Omega_X^d))},
    \end{align*}
    using the theorem of change of variables \ref{thm: change var}. Using Lemma \ref{lem:vois tub}, we then have:
    $$
    \mu_d(\underline{X}\smallsetminus\{O\} \cap S(O,n)) = \int_{\mathcal{V}(D)} \LL^{-v(\pi^*(\Omega_X^d))}.
    $$
    To compute local density, we need to normalize this volume by the volume of the $d$-dimensional sphere. Hence, we need to compute:
    $$
    \frac{\LL^{dn}}{1-\LL^{-d}}\int_{\mathcal{V}(D)} \LL^{-v(\pi^*(\Omega_X^d))}.
    $$
    As in the proof of the previous Lemma, in each point of the support of the divisor, we can define a finite Zariski cover as follow. In a point of $E_I^0$, working locally on $U$ an open set of $Y$, we can assume that $E_I^0$ is defined by the equation $\prod_{i\in I}x_i=0$. 
    By additivity of the motivic measure, it is sufficient to work on those local neighborhoods. The function $v(\pi^*(\Omega_X^d))$ can be expressed in terms of Mather discrepancies. Hence, near each divisor, the power of $\LL$ appearing in the integral is very easy to explicit. For $\mathcal{V}(D)$ defined locally by $v(x_1^{m_1}\dots x_d^{m_d})=n$, the function $v(\pi^*(\Omega_X^d))$ is equal to $-\hat{k}_1v(x_1)- \dots - \hat{k}_dv(x_d)$.

    \vspace{\baselineskip}

    We now define the Zeta function as:
    $$
    Z(T)=\sum_{I \subset J } Z_I(T),
    $$
    where the definition of $Z_I(T)$ is explained in what follows. As explained before, using the theorem of change of variables and expressing the Jacobian factor with the Mather discrepancies, and with $I=\{1,\dots,d\}$, we have to compute the following:

\begin{align*}
    &\frac{\LL^{dn}}{1-\LL^{-d}}\int_{v(x_1^{m_1}\dots {x_d}^{m_d})=n} \LL^{-\hat{k}_1v(x_1)- \dots -\hat{k}_dv(x_d)} \\
    =& \frac{\LL^{dn}}{1-\LL^{-d}} \sum_{\Lambda \subset I} \int_{\substack{ v(x_{i_1}^{m_{i_1}}\dots x_{i_d}^{m_{i_d}})=n \\ \Lambda=\{i_1,\dots,i_d\} \\ v(x_j)=0, ~j \notin \Lambda }} \LL^{-\sum \hat{k}_{i_j}v(x_{i_j})} \\
    =&\frac{\LL^{dn}}{1-\LL^{-d}}\sum_{\Lambda \subset I} \sum_{\substack{l_i \geq 1, ~i\in \Lambda \\ l_i=0,~i \notin \Lambda \\ \sum_{i \in \Lambda}m_il_i=n}} \LL^{-\sum_{i \in \Lambda}(\hat{k}_i+1)l_i} (\LL-1)^{|\Lambda|}\LL^{-d}[E_\Lambda^0\cap U] \\
    =& \frac{1}{1-\LL^{-d}}\sum_{\Lambda \subset I} \sum_{\substack{l_i \geq 1, ~i\in \Lambda \\ l_i=0,~i \notin \Lambda \\ \sum_{i \in \Lambda}m_il_i=n}} \LL^{-\sum_{i \in \Lambda}(\hat{k}_i+1-dm_i)l_i} (\LL-1)^{|\Lambda|}\LL^{-d}[E_\Lambda^0\cap U] \\
\end{align*}

We set $a_n^\Lambda=\displaystyle \sum_{\substack{l_i \geq 1, ~i\in \Lambda \\ l_i=0,~i \notin \Lambda \\ \sum_{i \in \Lambda}m_il_i=n}} \LL^{-\sum_{i \in \Lambda}(\hat{k}_i+1-d_im_i)l_i}$, and:
\begin{align*}
    Z_\Lambda(T)=& \sum_{n \geq 0}a_n^\Lambda T^n\\
    =& \sum_{l_i \geq 1, i \in \Lambda}\LL^{-\sum (\hat{k}_i+1-dm_i)l_i}T^{\sum_i m_il_i}\\
    =& \prod_{i \in \Lambda}\sum_{l_i \geq 1}(\LL^{-\alpha_i} T^{m_i})^{l_i}, \text{ where } \alpha_i=\hat{k}_i+1-dm_i \\
    =&\prod_{i \in \Lambda} \frac{\LL^{-\alpha_i}T^{m_i}}{1-\LL^{-\alpha_i}T^{m_i}}. \numberthis \label{eq:zeta}
\end{align*}

\vspace{\baselineskip}

    Since the sequences $(a_n^\Lambda)_n$ are positive constructible motivic functions, by Proposition-Definition 3.6 of  \cite{forey_motivic_2017}, the local density exists and the sequences $(a_n^\Lambda)_n$ admit a mean value at infinity for all $\Lambda$. Hence, using Theorem \ref{thm: zeta local volumes}, at most one of the $\alpha_i$ in (\ref{eq:zeta}) is equal to zero. 
    Locally, consider $\mathcal{V}(D)$ defined by $v(x_1^{m_1}\dots x_d^{m_d})=n$. If all of the $\alpha_i \neq 0$, then we are in the case where the Zeta function has no pole at 1, and the motivic local density is equal to zero. Suppose that there exists an $\alpha_i=0$, and suppose for simplicity that $i=1$. We only consider subsets $\Lambda=\{j_1,\dots,j_r\}, r \leq d$ of $J$ such that $j_1=1$, otherwise the motivic local density will be zero. Then:
    \begin{align*}
        Z_\Lambda(T)=& \frac{T^{m_1}}{1-T^{m_1}}\prod_{i \in \Lambda\setminus \{1\}} \frac{\LL^{-\alpha_i}T^{m_i}}{1-\LL^{-\alpha_i}T^{m_i}}.
    \end{align*}

    We now set $Z_I(T)=\sum_{\Lambda \subset I} Z_\Lambda(T)$. 
    
    Since:
    \begin{align*}
        \text{Res}_{t=1}Z_\Lambda(T)=& \frac{1}{m_1} \prod_{i \in \Lambda\setminus \{1\}} \frac{\LL^{-\alpha_i}}{1-\LL^{-\alpha_i}}  \\
        =& \text{MV}_\infty(a_n^\Lambda),
    \end{align*}

    then:
    \begin{align*}
        & \text{MV}_\infty\Big(\frac{1}{1-\LL^{-d}}\sum_{\Lambda \subset I} (\LL-1)^{|\Lambda|} \sum_{\substack{l_i \geq 1, ~i\in \Lambda \\ l_i=0,~i \notin \Lambda \\ \sum_{i \in \Lambda}m_il_i=n}} \LL^{-\sum_{i \in \Lambda}(\hat{k}_i+1-dm_i)l_i} \LL^{-d}[E_\Lambda^0\cap U]\Big) \\
        =&\text{MV}_\infty\Big( \frac{1}{1-\LL^{-d}}\sum_{\Lambda \subset I}(\LL-1)^{|\Lambda|}a_n^\Lambda\LL^{-d}[E_\Lambda^0\cap U]\Big)\\
        =&\frac{1}{1-\LL^{-d}} \sum_{\Lambda \subset I} \frac{1}{m_1}(\LL-1)^{|\Lambda|}\prod_{i \in \Lambda\setminus \{1\}} \frac{\LL^{-\alpha_i}}{1-\LL^{-\alpha_i}}\LL^{-d}[E_\Lambda^0\cap U]\Big)\\
        =& \frac{1}{m_1}\sum_{\Lambda \subset I}(\LL-1)^{|\Lambda|}\prod_{i \in \Lambda\setminus \{1\}} \frac{\LL^{-\alpha_i}}{1-\LL^{-\alpha_i}} \frac{[E_\Lambda^0\cap U]}{\LL^d-1}\\
        =&\text{Res}_{T=1}Z_I(T).
    \end{align*}

    \vspace{\baselineskip}
    
    By considering all subsets $I$ of $J$ such that for an $i, \alpha_i=0$, by additivity the result follow. 
\end{proof}
\vspace{\baselineskip}

\begin{remark}
    In dimension 2, this result had been proved in \cite{ratajczak_motivic_2026}. Let us briefly recall what are the inner rates related to the bilipschitz geometry of the singularity. Those inner rates were first introduced in \cite{birbrair_thick-thin_2014} by L. Birbrair, W. D. Neumann and A. Pichon, and were studied from the non-Archimedean view point in \cite{silva_inner_2022} by A. Belotto da Silva, L. Fantini and A. Pichon. They study the metric structure of the germ of an isolated complex surface singularity $(X, 0)$ by means of an infinite family of numerical analytic invariants, called inner rates. A resolution of $(X,0)$ is good if its exceptional locus is a normal crossing divisor. Given a good resolution $\pi:  Y \longrightarrow X$ of $(X,0)$ that factors through the blowup of the maximal ideal and through the Nash modification of $(X,0)$, an irreducible component $E_i$ of the exceptional divisor $\pi^{-1}(0)$ of $\pi$, and a small neighborhood $\mathcal{N}(E_i)$ of $E_i$ in $Y$ with a neighborhood of each double point of $\pi^{-1}(0)$ removed, the inner rate $q_i$ of $E_i$ is a rational number that measures how fast the subset $\pi(\mathcal{N}(E_i))$ of $(X,0)$ shrinks when approaching the origin. The inner rate $q_i$ is independent of the choice of an embedding of $(X,0)$ into a smooth germ, only depending on the analytic type of $(X,0)$ and on the divisor $E_i$. The finite set of inner rates which are bilipschitz invariants for the inner metric has been described explicitly in \cite{birbrair_thick-thin_2014}. It is worth noting that an inner rate equal to 1 characterizes a metrically conical region of the surface, it corresponds to linear inner contact between two complex curve germs on $(X,0)$.

    Since the inner rate associated to an irreducible component of the exceptional divisor is equal to $q_i=\frac{\hat{k}_i+1}{m_i}-1$ (see Proposition 4.8 of \cite{ratajczak_motivic_2026}), then $\alpha_i=0$ is equivalent to $q_i=1$. Moreover, $\alpha_j=m_j(q_j-1)$. If $I=\{i,i_2\}$ such that $E_i \cap E_{i_2} =\emptyset$, then $[E_I^0]=0$ and this term will not contribute to the sum. Hence, we find back the formula from the previous paper (See \cite{ratajczak_motivic_2026}):
    $$
    \Theta_2^\text{mot}(X,O)=\sum_{i \mid q_i=1}\frac{1}{m_i}\Big(\frac{[E_i^0]}{\LL+1}+\sum_{j \mid E_i \cap E_j \neq \emptyset} \frac{\LL^{-(q_j-1)m_j}}{1-\LL^{-(q_j-1)m_j}} \frac{\LL-1}{\LL+1}\Big).
    $$
\end{remark}

\vspace{\baselineskip}

\begin{definition}
    We define the formal inner rate associated to an irreducible component of the exceptional divisor $E_i$ of $\pi^{-1}(0)$ as $q_i:=\frac{\alpha_i}{m_i}+1$. In other words, $q_i=\frac{\hat{k}_i^\text{log}}{m_i}-d+1$. 
\end{definition}

\vspace{\baselineskip}

\begin{proposition}
    The formal inner rate associated to an irreducible component $E_i$ of $\pi^{-1}(0)$ satisfies $q_i \geq 1$. Consider a component $E_i$ of $\pi^{-1}(0)$ such that $q_i=1$. Then for all $j$ such that $E_i \cap E_j \neq \emptyset$, the associated formal inner rate $q_j >1$. 
\end{proposition}

\vspace{\baselineskip}

\begin{proof}
    The proof uses local density as in the previous section. Since we need to compute integrals like $\displaystyle \LL^{dn} \int_{v(x_i^{m_i})=n,v(x_j)=0, i\neq j}\LL^{-\hat{k}_iv(x_i)}$, then we get a term $\LL^{-(\hat{k}_i+1)\frac{n}{m_i}}\LL^{dn}$ where $m_i$ divides $n$ for example. So we get a power of $\LL$ in $-n(\frac{\hat{k}_i}{m_i}+\frac{1}{m_i}-d)$, and since the mean value at infinity exists, we need $\frac{\hat{k}_i}{m_i}+\frac{1}{m_i}-d \geq 0$, i.e. $\hat{k}_i+1 \geq dm_i$. From the proof of theorem \ref{thm:formula}, we already have that if $\alpha_i =0$, hence $\alpha_j \neq 0$ for all $j$ such that $E_i \cap E_j \neq \emptyset$. Therefore, if $q_i =1$, then $q_j \neq 1$. 
\end{proof}

\section{Link with usual motivic Zeta functions}

Zeta functions had been studied in various settings. We can cite the so-called Igusa Zeta function $Z_\text{Igusa}(f,s)$ \cite{Igusa1974}, defined as a $p$-adic parametric integral of a polynomial $f$ with coefficients in a $p$-adic field. Denef and Loeser defined in \cite{DenefLoeser1998} the invariant motivic (Igusa) Zeta function associated to a regular function on a nonsingular algebraic variety in characteristic zero. The latter specializes to the classical $p$-adic Igusa Zeta function. This definition had been generalized in \cite{VeysZunigaGalindo2008} where they associate to a sheaf of ideal $\mathcal{I}$ on a smooth algebraic variety (over a field of characteristic zero) a motivic Zeta function. One can also cite the work of W. Veys in \cite{Veys_2001}: on any $\Q-$Gorenstein variety, he associates a motivic Zeta function. Combining these different results, we define a motivic Zeta function associated to a sheaf of ideals on a singular variety. The previous definitions can be recovered by considering the corresponding setup. We can also apply this new definition to our setting of local density, and one finds back the Zeta function used in Section \ref{section: zeta function}, defined with normalized local volumes. 

\vspace{\baselineskip}

\begin{definition}
    Let $X$ be a $k-$variety of dimension $d$, and $\mathcal{I}$ a sheaf of ideal on $X$. For an element $\gamma \in \underline{X}$, we define $\text{ord}_t(\mathcal{I}(\gamma)):=\min_{f \in \mathcal{I}_{\gamma(0)}}\text{ord}_t(\gamma^* (f))$. This defines a simple function on $\underline{X}$, meaning that all of its fibers are measurable. 

    If $\mathcal{I}$ is locally generated by $(f_1,\dots,f_r)$, then
    $$
    \text{ord}_t\mathcal{I}(\gamma)=\min_{1 \leq i \leq r} \text{ord}_t(f_i(\gamma)).
    $$ 
\end{definition}

\vspace{\baselineskip}

\begin{definition}\label{def: big zeta}
    Let $X$ be an affine $k$-variety of dimension $d$, and $\mathcal{I}$ be a sheaf of ideals on $X$. Define 
    $$
    Z_{X,\mathcal{I}}(T):=\sum_{n \geq 0} \mu_d(\{ \gamma \in \underline{X} \mid \text{ord}_t\mathcal{I}(\gamma)=n\})T^n,
    $$ 
    Informally, this is equivalent to defining the following integral:
    $$
    \int_{\underline{X}} |\mathcal{I}(\gamma)|^s d\mu_d,
    $$
    where $|\mathcal{I}(\gamma)|=\LL^{-\text{ord}_t(\mathcal{I}(\gamma))}$, $T=\LL^{-s}$ and the volume form is coming from the affine variety (See Section \ref{section : mot int}).
\end{definition}

\vspace{\baselineskip}

\begin{proposition}
    The Zeta function $Z_{X,\mathcal{I}}(T)$ is an element of $\mathcal{M}_k[T,(1-\LL^{-a}T^b)^{-1}], a\in \Z, b \geq 1$.
\end{proposition}

\vspace{\baselineskip}

\begin{proof}
    Set $Z:=\{(\gamma,n)\in \underline{X}\times \N \mid \text{ord}_t\mathcal{I}(\gamma)=n\}$, and $\pi : Z \longrightarrow \N, (\gamma,n) \mapsto n$. Using the rationality results from \cite{cluckers_constructible_2008}, see Theorem 14.4.1, we can conclude that $\mu_d(\mathbf{1}_{Z_{|\pi^{-1}(n)}})$ is a constructible motivic function and 
    $$
    \sum_{n \geq 0} \mu_d(\{ \gamma \in \underline{X} \mid \text{ord}_t\mathcal{I}(\gamma)=n\})T^n
    $$
    lies in the ring $\mathcal{M}_k[T,(1-\LL^{-a}T^b)^{-1}]$, $a\in \Z, b \geq 1$.
\end{proof}

\vspace{\baselineskip}

In our special case for computing the motivic local density, we set $k=\C$ and $I=\mathfrak{M}$ the maximal ideal. Since
$$
\mu_d(\gamma \in \underline{X} \mid \text{ord}_t\mathfrak{M}(\gamma)=n)=\mu_d(\underline{X} \cap S(O,n)),
$$
we have the following equality:

\begin{align*}
    \frac{Z_{X,\mathfrak{M}}(T\LL^d)}{1-\LL^{-d}} & = \sum_{n \geq 0} \frac{\mu_d(\underline{X} \cap S(O,n))\LL^{dn}}{1-\LL^{-d}}T^n \\
    &= \sum_{n \geq 0} \frac{\mu_d(\underline{X}\cap S(O,n))}{\mu_d(S(O,n))}T^n,
\end{align*}
and we find back the Zeta function used in Section \ref{section: zeta function}, with local volumes defined with spheres.

\vspace{\baselineskip}

\begin{remark}
    Using the proof of Theorem \ref{thm:formula}, for $\pi : Y \longrightarrow X$ a good resolution of the singularity, $E_i, i  \in J$ the irreducible components of the exceptional locus, $E_I=\cap_{i \in I} E_i, E_I^0=E_I \setminus \cup_{j \in J \setminus I} E_j$, we can deduce that
    $$
    \frac{Z_{X,\mathfrak{M}}(T\LL^d)}{1-\LL^{-d}}=\frac{1}{\LL^d-1}\sum_{I \subset J}[E_I^0] \prod_{i \in I} \frac{(\LL-1)\LL^{-\hat{k}_i-1+dm_i}T^{m_i}}{1-\LL^{-\hat{k}_i-1+dm_i}T^{m_i}}.
    $$
\end{remark}

\vspace{\baselineskip}

From the definition \ref{def: big zeta}, we can recover the definition of the Zeta function defined in \cite{VeysZunigaGalindo2008}. Denote for $n \in \mathbb{N}$ by $\mathfrak{X}_{n}$ the variety of $(n+1)-$jets $\gamma$ on $X$ for which $\text{ord}_t(\gamma^*\mathcal{I})=n$. In the smooth case, since $\mu_d(\{ \gamma \in X(k[t]/(t^{n+1})) \mid \text{ord}_\mathcal{I}(\gamma)=n\})=[\mathfrak{X}_{n}]\LL^{-(n+1)d}$, one finds back the definition of the motivic Zeta function associated to a sheaf of ideals, from \cite{VeysZunigaGalindo2008}, up to a constant factor:

\vspace{\baselineskip}

\begin{definition}[Definition 2.16 of \cite{VeysZunigaGalindo2008}]
    Let $X$ be a smooth algebraic variety of dimension $d$ over a field $k$ of characteristic zero, and $\mathcal{I}$ a sheaf of ideals on $X$. The motivic Zeta function associated to $\mathcal{I}$ is the formal power series 
    $$
    Z(\mathcal{I},T)=\LL^{-d} \sum_{n \geq 0} [\mathfrak{X}_{n}](\LL^{-d}T)^n.
    $$
\end{definition}

\vspace{\baselineskip}

\begin{remark}
    Using a log-principalization of the sheaf of ideals, they also get a formula for the Zeta function, depending on the numerical data of the resolution and the exceptional components. We recall this theorem.
\end{remark}

\vspace{\baselineskip}

\begin{theorem} [Theorem 2.17 of \cite{VeysZunigaGalindo2008}]
    Let $\sigma : Y \longrightarrow X$ be a log-principalization of $\mathcal{I}$. With the notations $E_i,N_i,\nu_i, i \in J$, where $(N_i,\nu_i), i \in J$ are the numerical data of $\sigma$, and $E_I^0$ for $I \subset J$, we have:
    $$
    Z(\mathcal{I},T)=\LL^{-d}\sum_{I \subset J} [E_I^0]\prod_{i \in I}\frac{(\LL-1)T^{N_i}}{\LL^{\nu_i}-T^{N_i}}.
    $$
\end{theorem}

\vspace{\baselineskip}

In the special case where $\mathcal{I}=(f)$ for a morphism $f: X \longrightarrow \mathbb{A}^1_\C$, one can recover the results of \cite{DenefLoeser1998} from our definition \ref{def: big zeta}. We recall this setting. Denote $X_{n,f}$ the reduced subscheme of $X(k[t]/(t^{n+1}))$ defined by $X_{n,f}=\{ \gamma \in X(k[t]/(t^{n+1})) \mid \text{ord}_t(f(\gamma))=n\}$. Then they define (See Section 2 of \cite{DenefLoeser1998}):
$$
\int f^s :=\sum_{n\geq 0} [X_{n,f}]\LL^{-(n+1)d}T^n,
$$
which is equal to $Z_{X,(f)}(T)$ from our definition \ref{def: big zeta} in this setting.

\vspace{\baselineskip}




\bibliographystyle{plain}
\bibliography{Local_density_varieties}

\end{document}